\documentclass[11pt, leqno, twoside]{article}

\usepackage{amssymb}
\usepackage{amsmath}
\usepackage{amsthm}
\usepackage{amsfonts}
\usepackage{mathrsfs}
\usepackage{color}

\allowdisplaybreaks

\def\rr{{\mathbb R}}
\def\rn{{{\rr}^n}}

\def\zz{{\mathbb Z}}
\def\cc{{\mathbb C}}
\def\nn{{\mathbb N}}
\def\cp{{\mathcal P}}

\def\cf{{\mathcal F}}

\def\cm{{\mathcal M}}

\def\fz{\infty}
\def\az{\alpha}
\def\bz{\beta}

\def\ez{\epsilon}

\def\lz{\lambda}

\def\boz{{\Omega}}

\def\lf{\left}
\def\r{\right}

\def\hs{\hspace{0.26cm}}

\def\ls{\lesssim}

\def\noz{\nonumber}
\def\wz{\widetilde}
\def\wh{\widehat}
\def\st{\subset}
\def\com{\complement}
\def\bh{\backslash}

\def\cs{{\mathcal S}}

\def\gfz{\genfrac{}{}{0pt}{}}

\def\dist{\mathop\mathrm{\,dist\,}}
\def\supp{\mathop\mathrm{\,supp\,}}

\def\q1{\wz q}
\def\Q1{q_1}
\def\lv{{L^{p(\cdot)}(\rn)}}
\def\wlv{W\!L^{p(\cdot)}(\rn)}
\def\hv{{H^{p(\cdot)}(\rn)}}

\def\whv{{W\!H^{p(\cdot)}(\rn)}}

\def\bije{{b_{i,j}^\ez}}

\def\qij{{Q_{i,j}}}
\def\wqij{{\wz Q_{i,j}}}

\def\dint{\displaystyle\int}

\def\dsup{\displaystyle\sup}

\def\supp{{\mathop\mathrm{\,supp\,}}}

\def\dist{{\mathop\mathrm{\,dist\,}}}

\newtheorem{thm}{Theorem}[section]
\newtheorem{prop}[thm]{Proposition}
\newtheorem{lem}[thm]{Lemma}
\newtheorem{cor}[thm]{Corollary}
\theoremstyle{definition}
\newtheorem{defn}[thm]{Definition}
\newtheorem{rem}[thm]{Remark}

\numberwithin{equation}{section}
\numberwithin{equation}{section}

\begin{document}

\title{\bf\Large Interpolation between $H^{p(\cdot)}(\mathbb R^n)$
and $L^\infty(\mathbb R^n)$: Real Method
\footnotetext{\hspace{-0.35cm} 2010 {\it
Mathematics Subject Classification}. Primary 42B30;
Secondary 42B35, 46B70.
\endgraf {\it Key words and phrases.} (weak) Hardy space, (weak) Lebesgue space, variable exponent, real interpolation.
\endgraf The first author is supported by the Construct Program of the Key Discipline in
Hunan Province. This project is also supported by the National
Natural Science Foundation of China
(Grant Nos. 11571039, 11671185 and 11471042).}}
\author{Ciqiang Zhuo, Dachun Yang\,\footnote{Corresponding author / January 1, 2017.}\ \ and Wen Yuan}
\date{ }
\maketitle

\vspace{-0.8cm}

\begin{center}
\begin{minipage}{13cm}
{\small {\bf Abstract}\quad
Let $p(\cdot):\ \mathbb R^n\to(0,\infty)$ be a variable exponent
function satisfying the globally log-H\"older continuous condition.
In this article, the authors first obtain a decomposition for any distribution
of the variable weak Hardy space into ``good" and ``bad" parts
and then prove the following real interpolation theorem between
the variable Hardy space $H^{p(\cdot)}(\mathbb R^n)$ and the space $L^{\infty}(\mathbb R^n)$:
\begin{equation*}
(H^{p(\cdot)}(\mathbb R^n),L^{\infty}(\mathbb R^n))_{\theta,\infty}=
W\!H^{p(\cdot)/(1-\theta)}(\mathbb R^n),\quad \theta\in(0,1),
\end{equation*}
where $W\!H^{p(\cdot)/(1-\theta)}(\mathbb R^n)$ denotes the variable
weak Hardy space. As an application, the variable weak Hardy space
$W\!H^{p(\cdot)}(\mathbb R^n)$ with $p_-:=\mathop\mathrm{ess\,inf}_{x\in\rn}p(x)\in(1,\infty)$
is proved to coincide with the variable Lebesgue space
$W\!L^{p(\cdot)}(\mathbb R^n)$.
}
\end{minipage}
\end{center}

\section{Introduction\label{s-intro}}
\hskip\parindent
In recent years, theories of several variable function spaces, based on the variable
Lebesgue space, have been rapidly developed (see, for example, \cite{aa16,ah10,cw14,dhr09,ns12,
Xu08,yyyz16,yzy15,yzy151}).
Recall that the variable Lebesgue space $\lv$, with a variable exponent function
$p(\cdot):\ \rn\to(0,\fz)$, is a generalization of the classical Lebesgue space $L^p(\rn)$.
The study of variable Lebesgue spaces can be traced back to Orlicz \cite{or31},
moreover, they have been the subject of more intensive study since the early
work \cite{kr91} of Kov\'a\v{c}ik and R\'akosn\'{\i}k and \cite{fz01} of Fan and Zhao
as well as \cite{cruz03} of Cruz-Uribe and \cite{din04} of Diening, because of their intrinsic
interest for applications into harmonic analysis, partial differential equations and variational
integrals with nonstandard growth conditions (see also \cite{am02,am05,ins14,yyz16} and their references).

As a generalization of the classical Hardy space $H^{p}(\rn)$ and the
variable Lebesgue space $\lv$, the variable Hardy spaces $\hv$ were first investigated by
Nakai and Sawano \cite{ns12} with $p(\cdot)$ satisfying the globally log-H\"older continuous
condition. In \cite{ns12}, they established the atomic characterizations
of $\hv$, which were further applied to consider their dual spaces and the boundedness of
singular integral operators on $\hv$. Later, Sawano \cite{Sa13} extended and improved the atomic
characterization of $\hv$ in \cite{ns12} and Zhuo et al. \cite{zyl16} gave their equivalent
characterizations via (intrinsic) square functions including
the (intrinsic) Lusin-area function, the (intrinsic) Littlewood-Paley $g$-function
or $g_\lz^\ast$-function. Independently, Cruz-Uribe and Wang \cite{cw14} also studied
the variable Hardy spaces $\hv$ with $p(\cdot)$ satisfying
some conditions slightly weaker than those used in \cite{ns12}, and established their
equivalent characterizations by means of radial or non-tangential maximal functions
or atoms. However, the atomic
characterization of $\hv$ obtained in \cite{cw14} is very different
from the classical case, which, in spirit,
is   closer to the atomic characterization for weighted Hardy spaces due to
Str\"omberg and Torchinsky \cite{StTo89}. In addition,
the characterizations of $\hv$ via Riesz transforms with $p(\cdot)$
satisfying the same conditions as in \cite{cw14} were presented in \cite{yzn16}.

Very recently, motivated by the well-known fact that, when studying the boundedness of
some singular integral operators in the critical case, the  {weak Hardy space} $W\!H^p(\rn)$, with any
$p\in(0,1]$, naturally appears as a proper substitute of the Hardy space $H^p(\rn)$ (see \cite{fs86,liu91}),
Yan et al. \cite{yyyz16} introduced the variable weak Hardy space $\whv$ and characterized
these spaces via  the radial or the
non-tangential maximal functions, atoms, molecules, the Lusin-area function,
the Littlewood-Paley $g$-function or $g_\lz^\ast$-function.
As an application, the authors in \cite{yyyz16} established the boundedness of some
convolutional $\delta$-type and non-convolutional $\gamma$-order Calder\'on-Zygmund
operators from $\hv$ to $\whv$ including the critical case when
$p_-=\frac n{n+\delta}$ or $p_-=\frac n{n+\gamma}$, where
\begin{equation}\label{2.1x}
p_-:=\mathop\mathrm{ess\,inf}_{x\in \rn}p(x),
\end{equation}
which implies that the space $\whv$ is a suitable substitute of the space $\hv$
in the study of boundedness of some singular integral operators in the critical case
on $\hv$.

As was well known,  Fefferman et al. \cite{frs74} found that the weak
Hardy space $W\!H^p(\rn)$   naturally appears as the intermediate space
of the classical Hardy space $H^p(\rn)$ and the space $L^\fz(\rn)$
under the real interpolation, which is
another main motivation to develop the real-variable theory of $W\!H^p(\rn)$.
Therefore, it is natural and interesting to ask whether or not the variable 
weak Hardy space serves as the intermediate space
between the variable Hardy space $\hv$ and the space $L^\fz(\rn)$ via the real interpolation.

On the other hand, it is well known that,
when $p\in(1,\fz)$,
\begin{equation}\label{1219f}
W\!H^p(\rn)=W\!L^p(\rn)
\end{equation}
with equivalent quasi-norms
(see \cite{frs74}), where $W\!L^p(\rn)$ denotes the \emph{classical weak Lebesgue space}.
Thus, it is also interesting to know whether or not
this coincidence \eqref{1219f} remains true in the variable setting
under some restriction on the variable exponent function.

In this article, we   give positive answers to the above two questions.
Indeed, in Theorem \ref{t-inter} below,
we prove that the real interpolation space between the variable Hardy space and the space
$L^\fz(\rn)$ is just the variable weak Hardy space introduced in \cite{yyyz16},
via first establishing a useful  decomposition for any distribution of the variable weak Hardy space
into ``good" and ``bad" parts (see Proposition \ref{p-czd} below).
As an application, we conclude that, when $p_-\in(1,\fz)$,
the variable weak Hardy space $\whv$ coincides with the variable weak Lebesgue space
$\wlv$.

To state the main result of this article, we first
recall some basic notions about the theory of real interpolation (see \cite{bl76}).
Let $(X_0,X_1)$ be a compatible couple of quasi-normed spaces, namely, $X_0$ and $X_1$
are two quasi-normed linear spaces which are continuously
embedded into some large topological vector space.
Let
$$X_0+X_1:=\lf\{f_0+f_1:\ f_0\in X_0\ {\rm and}\ f_1\in X_1\r\}.$$
For any $t\in(0,\fz)$, the \emph{Peetre $K$-functional} $K(t,f;X_0,X_1)$ on $X_0+X_1$ is defined
by setting, for any $f\in X_0+X_1$,
\begin{equation*}
K(t,f;X_0,X_1):=\inf\{\|f_0\|_{X_0}+t\|f_1\|_{X_1}:\
f=f_0+f_1,\ f_0\in X_0\ {\rm and}\ f_1\in X_1\}.
\end{equation*}
Then, for any $\theta\in(0,1)$ and $q\in(0,\fz]$, the \emph{real interpolation
space} $(X_0,X_1)_{\theta,q}$ between $X_0$ and $X_1$ is defined as
$$(X_0,X_1)_{\theta,q}:=\lf\{f\in X_0+X_1:\
\|f\|_{\theta,q}<\fz\r\},$$
where, for any $f\in X_0+X_1$,
\begin{equation*}
\|f\|_{\theta,q}:=
\begin{cases}
\lf[\dint_0^\fz \lf\{t^{-\theta}K(t,f;X_0,X_1)\r\}^q\,\frac{dt}t\r]^{1/q}
\quad & {\rm if}\ q\in(0,\fz),\\
\dsup_{t\in(0,\fz)}t^{-\theta}K(t,f;X_0,X_1)\quad & {\rm if}\ q=\fz.
\end{cases}
\end{equation*}

We also recall some notation about variable Lebesgue spaces.
For a detailed exposition of these concepts,
we refer the reader to the monographs \cite{cfbook,dhr11}.
Denote by $\cp(\rn)$ the \emph{collection of all variable exponent functions}
$p(\cdot):\ \rn\to(0,\fz)$ satisfying
\begin{align*}
0<p_-\le p_+:=
\mathop\mathrm{ess\,sup}_{x\in \rn}p(x)<\fz,
\end{align*}
where $p_-$ is as in \eqref{2.1x}.
For a measurable function $f$ on $\rn$ and $p(\cdot)\in\cp(\rn)$,
the \emph{modular functional} (or, simply, the \emph{modular})
$\varrho_{p(\cdot)}$, associated with $p(\cdot)$, is defined by setting
$$\varrho_{p(\cdot)}(f):=\int_\rn|f(x)|^{p(x)}\,dx$$ and the
\emph{Luxemburg} (also known as the \emph{Luxemburg-Nakano})
\emph{quasi-norm} is given by setting
\begin{equation*}
\|f\|_{\lv}:=\inf\lf\{\lz\in(0,\fz):\ \varrho_{p(\cdot)}(f/\lz)\le1\r\}.
\end{equation*}

\begin{defn}
Let $p(\cdot)\in\cp(\rn)$.
\begin{enumerate}
\item[(i)] The \emph{variable Lebesgue space} $\lv$ is defined to be the
set of all measurable functions $f$ on $\rn$ such that the (quasi-)norm $\|f\|_{\lv}$
is finite.

\item[(ii)] The \emph{variable weak Lebesgue space} $W\!L^{p(\cdot)}(\rn)$
is defined to be the set of all measurable functions $f$ on $\rn$ such that
\begin{align*}
\|f\|_{W\!L^{p(\cdot)}(\rn)}:=\sup_{\az\in(0,\fz)}\az\lf
\|\chi_{\{x\in\rn:\ |f(x)|>\az\}}\r\|_{L^{p(\cdot)}(\rn)}<\fz.
\end{align*}
\end{enumerate}
\end{defn}

\begin{rem}\label{r-lv}
Let $p(\cdot)\in\cp(\rn)$ and $s\in(0,\fz)$.
\begin{enumerate}
\item[(i)] It is easy to see that, for any $f\in\lv$,
$\lf\||f|^s\r\|_{\lv}=\|f\|_{L^{sp(\cdot)}(\rn)}^s$.
Moreover, for any $\lz\in{\mathbb C}$ and $f,\ g\in\lv$,
$\|\lz f\|_{\lv}=|\lz|\|f\|_{\lv}$ and
$$\|f+g\|_{\lv}^{\underline{p}}\le \|f\|_{\lv}^{\underline{p}}
+\|g\|_{\lv}^{\underline{p}},$$
here and hereafter,
\begin{align}\label{1219a}
\underline{p}:=\min\{p_-,1\}
\end{align}
with $p_-$ as in \eqref{2.1x}.
Particularly, when $p_-\in[1,\fz)$,
$\lv$ is a Banach space (see \cite[Theorem 3.2.7]{dhr11}).

\item[(ii)] For any $f\in \wlv$, we have
$\||f|^s\|_{\wlv}=\|f\|_{W\!L^{sp(\cdot)}(\rn)}^s$
(see \cite[Lemma 2.11]{yyyz16}) and it was proved in \cite[Lemma 2.9]{yyyz16} that, for any
$\lz\in(0,\fz)$ and $f,\ g\in \wlv$, $\|\lz f\|_{\wlv}=|\lz|\|f\|_{\wlv}$ and
$$\|f+g\|_{\wlv}^{\underline{p}}\le 2^{\underline{p}}
\lf[\|f\|_{\wlv}^{\underline{p}}+\|g\|_{\wlv}^{\underline{p}}\r].$$
\end{enumerate}
\end{rem}

A function $p(\cdot)\in\cp(\rn)$ is said to satisfy the
\emph{globally log-H\"older continuous condition}, denoted by $p(\cdot)\in C^{\log}(\rn)$,
if there exist positive constants $C_{\log}(p)$ and $C_\fz$, and
$p_\fz\in\rr$ such that, for any $x,\ y\in\rn$,
\begin{equation*}
|p(x)-p(y)|\le \frac{C_{\log}(p)}{\log(e+1/|x-y|)}
\end{equation*}
and
\begin{equation*}
|p(x)-p_\fz|\le \frac{C_\fz}{\log(e+|x|)}.
\end{equation*}

In what follows, denote by $\cs(\rn)$ the \emph{space of all
Schwartz functions} on $\rn$ equipped with the well-known classical topology
and $\cs'(\rn)$ its \emph{topological dual space}
equipped with the weak-$*$ topology.
For any $N\in\nn$, let
\begin{align*}
\cf_N(\rn):=\lf\{\psi\in\cs(\rn):\ \sum_{\bz\in\zz_+^n,\,|\bz|\le N}
\sup_{x\in\rn}\lf[(1+|x|)^N|D^\bz\psi(x)|\r]\le1\r\},
\end{align*}
where, for any $\bz:=(\bz_1,\dots,\bz_n)\in\zz_+^n$,
$|\bz|:=\bz_1+\cdots+\bz_n$ and
$D^\bz:=(\frac\partial{\partial x_1})^{\bz_1}
\cdots(\frac\partial{\partial x_n})^{\bz_n}$.
Then, for any $f\in\cs'(\rn)$, the \emph{radial grand maximal function}
$f^\ast_{N,+}$
of $f$ is defined by setting, for any $x\in\rn$,
\begin{equation}\label{2.8x}
f_{N,+}^\ast(x):=\sup\lf\{|f\ast\psi_t(x)|:\
t\in(0,\fz)\ {\rm and}\ \psi\in\cf_N(\rn)\r\},
\end{equation}
where, for any $t\in(0,\fz)$ and $\xi\in\rn$,
$\psi_t(\xi):=t^{-n}\psi(\xi/t)$.

\begin{rem}\label{r-1214}
For any  $a\in(0,\fz)$, $N\in\nn$, $f\in \cs'(\rn)$ and $x\in\rn$, let
\begin{equation*}
f_{N,\triangledown,a}^\ast(x):=\sup_{\psi\in\cf_N(\rn)}\sup\lf\{|f\ast \psi_t(y)|:\ y\in\rn,\ t\in(0,\fz),\ |y-x|<at\r\}.
\end{equation*}
Then, by an argument similar to that used in the proof of \cite[Proposition 2.1]{zyl16}, we know
that $f_{N,+}^\ast\sim f_{N,\triangledown,a}^\ast$ with the
equivalent positive constants independent of $f$.
\end{rem}

Now we recall the definitions of both the variable Hardy space from Nakai and Sawano
\cite{ns12} and the variable weak Hardy space from Yan et al. \cite{yyyz16}.

\begin{defn}\label{d-hardy}
Let $p(\cdot)\in C^{\log}(\rn)$ and
$N\in(\frac{n}{\underline{p}}+n+1,\fz)$ be a positive integer,
where $\underline{p}$ is as in \eqref{1219a}.
\begin{enumerate}
\item[(i)] The \emph{variable Hardy space} $\hv$ is defined to be the set of all
$f\in\cs'(\rn)$ such that $f_{N,+}^\ast\in \lv$, equipped with the (quasi-)norm
$$\|f\|_{\hv}:=\|f_{N,+}^\ast\|_{\lv}.$$

\item[(ii)] The \emph{variable weak Hardy space} $\whv$ is defined to be the set of all
$f\in\cs'(\rn)$ such that
$f_{N,+}^\ast\in\wlv$, equipped with the {quasi-norm}
$$\|f\|_{\whv}:=\|f_{N,+}^\ast\|_{\wlv}.$$
\end{enumerate}
\end{defn}

The main result of this article is stated as follows.

\begin{thm}\label{t-inter}
Let $p(\cdot)\in C^{\log}(\rn)$ and $\theta\in(0,1)$.
Then it holds true that
\begin{equation}\label{0928x}
(H^{p(\cdot)}(\rn), L^\fz(\rn))_{\theta,\fz}=W\!H^{\wz p(\cdot)}(\rn),
\end{equation}
where $\frac1{\wz p(\cdot)}=\frac{1-\theta}{p(\cdot)}$.
\end{thm}

As a consequence of Theorem \ref{t-inter} and \cite[Lemma 3.1]{ns12}, we immediately obtain
the following conclusion.

\begin{cor}\label{cor}
Let $p(\cdot)\in C^{\log}(\rn)$. If $p_-\in(1,\fz)$, then
$\whv=\wlv$
with equivalent quasi-norm.
\end{cor}

\begin{rem}\label{r-inter}
\begin{enumerate}
\item[(i)] When $p(\cdot)\equiv p\in(0,1)$, Theorem \ref{t-inter} goes back
to \cite[Theorem 1]{frs74}, which states that
$$(H^{p}(\rn),L^\fz(\rn))_{\theta,\fz}=W\!H^{p/(1-\theta)}(\rn),\quad \theta\in(0,1).$$
\item[(ii)] When $p(\cdot)\equiv 1$, \eqref{0928x} becomes
$$(H^1(\rn),L^\fz(\rn))_{\theta,\fz}=W\!H^{1/(1-\theta)}(\rn)
=W\!L^{1/(1-\theta)}(\rn),\quad \theta\in(0,1),$$
which was presented in \cite[(2)]{rs73}.

\item[(iii)] When $p(\cdot)\equiv p\in(1,\fz)$, \eqref{0928x} is a special case
of \cite[Theorem 7]{rs73}, namely,
$$(L^p(\rn),L^\fz(\rn))_{\theta,\fz}=W\!L^{p/(1-\theta)}(\rn),\quad \theta\in(0,1).$$
\end{enumerate}
\end{rem}

The proofs of   Theorem \ref{t-inter} and Corollary \ref{cor} are presented in
Section \ref{s-2}.

The main and key step in the proof of Theorem \ref{t-inter}
is to  decompose any distribution $f$ from the variable weak Hardy space
$W\!H^{\wz p(\cdot)}(\rn)$ into ``good" and ``bad" parts
(see Proposition \ref{p-czd} below). The vector-valued inequality of
the Hardy-Littlewood maximal function on the variable Lebesgue
space, obtained by Cruz-Uribe and Fiorenza \cite[Corollary 2.1]{cf06}
(see also Lemma \ref{l1230-a} below),
and the atomic characterization of $H^{p(\cdot)}(\rn)$ via
$(p(\cdot),\fz)$-atoms established by Nakai and Sawano \cite[Theorem 4.5]{ns12}
(see also Lemma \ref{l1230b} below),
play the key roles in the proof of Proposition \ref{p-czd}.
By using this proposition and some ideas from the proof of
\cite[Theorem 4.1]{kv14}, we further prove that
\begin{equation}\label{1212a}
W\!H^{\wz p(\cdot)}(\rn)\st (H^{p(\cdot)}(\rn), L^\fz(\rn))_{\theta,\fz},
\end{equation}
where $\theta,\ p(\cdot)$ and $\wz p(\cdot)$ are as in Theorem \ref{t-inter}.
The converse part of \eqref{1212a} is proved
by applying the sublinear operator $T(f):=f_{N,+}^\ast$,
where $N$ is as in Definition \ref{d-hardy},
to the real interpolation property between $L^{p(\cdot)}(\rn)$ and $L^\fz(\rn)$,
which is a special case of \cite[Theorem 4.1]{kv14} of Kempka and Vyb\'iral
when $p(\cdot)=q_0(\cdot)$ and $q=\fz$ (see also Lemma \ref{l1230c} below).
Applying this last real interpolation property, between $L^{p(\cdot)}(\rn)$ and $L^\fz(\rn)$,
and Theorem \ref{t-inter}, we immediately obtain Corollary \ref{cor}.

Here we point out that the approach used in the proof of \eqref{1212a} is
quite different from that used in the proof of \cite[Theorem 1]{frs74}
(see Remark \ref{r-inter}(i) below).
Indeed,  in \cite[Theorem 1]{frs74}, it seems to be only proved that the embedding
$$\lf[W\!H^{\wz p}(\rn)\cap \cs(\rn)\r]\st (H^{p}(\rn),L^\fz(\rn))_{\theta,\fz}$$
holds true instead of $W\!H^{\wz p}(\rn)\st (H^{p}(\rn),L^\fz(\rn))_{\theta,\fz}$,
since  the proof of \cite[Theorem 1]{frs74} strongly depends on
the decomposition obtained in \cite[Lemma A]{frs74} (see also Remark \ref{r-decom} below),
which was proved only for any $f\in \cs(\rn)$ instead of all distributions
and, more importantly, the Schwartz class is not dense in the space $W\!H^{\wz p}(\rn)$
(see, for example, \cite{gh16,He14}).
To overcome this gap and difficulty, in this article,
we establish a decomposition of any distribution of $f\in W\!H^{\wz p(\cdot)}(\rn)$ into
``good" and ``bad" parts in Proposition \ref{p-czd}
via a modification of technique due to Calder\'on \cite{c77}
and some ideas from the proof of \cite[Theorem 4.4]{yyyz16} in which the
atomic characterizations of variable weak Hardy spaces were obtained.

Finally, we make some conventions on notation. Let $\nn:=\{1,2,\dots\}$,
$\zz_+:=\nn\cup\{0\}$ and $\rr_+^{n+1}:=\rn\times(0,\fz)$.
We denote by $C$ a \emph{positive constant}
which is independent of the main parameters,
but may vary from line to line. The \emph{symbol}
$f\ls g$ means $f\le Cg$. If $f\ls g$ and $g\ls f$, we then write $f\sim g$.
If $E$ is a subset of $\rn$, we denote by $\chi_E$ its
\emph{characteristic function} and by $E^\complement$
the set $\rn\backslash E$.
For any $x\in\rn$ and $r\in(0,\fz)$, denote by $Q(x,r)$
the cube centered at $x$ with side length $r$, whose sides are parallel to the axes
of coordinates. For each cube $Q\st\rn$, we use $x_Q$ to denote its center
and $\ell(Q)$ its side length and, for any $a\in(0,\fz)$, denote by $aQ$
the cube concentric with $Q$ having the side length $a\ell(Q)$.
We use $\vec0_n$ to denote the \emph{origin} of $\rn$. For any $a\in\rr$, let $\lfloor a\rfloor$
be the maximal integer not bigger than $a$.

\section{A key decomposition\label{s2-x}}

\hskip\parindent
In this section, we aim to establish a decomposition
of any distribution $f$ belonging to
$W\!H^{p(\cdot)/(1-\theta)}(\rn)$
into ``good" and
``bad" parts in Proposition \ref{p-czd} below, which plays a key role in the
proofs of Theorem \ref{t-inter} and Corollary \ref{cor} in Section \ref{s-2}.

We begin with some notation. For any $p(\cdot)\in \cp(\rn)$,
let $\psi\in\cs(\rn)$ be such that $\supp \psi\st B(\vec 0_n,1)$, $\int_\rn\psi(x) x^\gamma\,dx=0$
for any $\gamma\in\zz_+^n$ with $|\gamma|\le s_0$,
where
\begin{equation}\label{1218a}
s_0:=\lfloor n(1/p_--1)\rfloor\ {\rm with}\ p_-\ {\rm as\ in}\ \eqref{2.1x}.
\end{equation}
Then there exists $\phi\in\cs(\rn)$ satisfying
that $\wh\phi$ has compact support away from the origin and, for any
$x\in \rn\backslash\{\vec 0_n\}$,
\begin{equation*}
\int_0^\fz\wh\psi(tx)\wh\phi(tx)\,\frac{dt}t=1;
\end{equation*}
see, for example, \cite[(3.1)]{c77}.
Define a function $\eta$ on $\rn$ by setting, for any $x\in \rn \backslash\{\vec 0_n\}$,
$$\wh\eta(x):=\int_1^\fz\wh\psi(tx)\wh\phi(tx)\,\frac{dt}t$$
and $\wh\eta(\vec 0_n):=1$. Then $\eta$ is infinitely differentiable,
has compact support and equals 1 near the origin (see \cite[p.\,219]{c77}).
Moreover, for any $t_0,\ t_1\in(0,\fz)$ and $x\in \rn \backslash\{\vec 0_n\}$,
\begin{equation}\label{1018-y}
\int_{t_0}^{t_1}\wh\psi(tx)\wh\phi(tx)\,\frac{dt}t=\wh\eta(t_0x)-\wh\eta(t_1x).
\end{equation}

Let $x_0:=(\overbrace{2,\,...\,,2}^{n\,{\rm times}})\in\rn$ and $f\in \whv$.
For any $x\in\rn$, let
$$\wz\phi(x):=\phi(x-x_0),\quad \wz\psi(x):=\psi(x+x_0)$$
and, for any $t\in(0,\fz)$,
\begin{equation}\label{fandg}
F(x,t):=f\ast \wz\phi_t(x),\quad G(x,t):=f\ast \eta_t(x).
\end{equation}
Then, using \eqref{1018-y}, we have, for any $t_0,\ t_1\in(0,\fz)$ and
$x\in \rn \backslash\{\vec 0_n\}$,
\begin{equation}\label{1019-z}
\int_{t_0}^{t_1}\int_\rn F(y,t)\wz \psi(x-y)\,\frac{dydt}{t}=G(x,t_0)-G(x,t_1),
\end{equation}
and, by \cite[p.\,220]{c77} and the fact that $f\in\cs'(\rn)$, we know that
\begin{equation*}
f=\lim_{\gfz{\ez\to0}{\delta\to\fz}}\int_{\ez}^\delta
\int_\rn F(y,t)\wz \psi (\cdot-y)\,\frac{dydt}{t}
\end{equation*}
converges in $\cs'(\rn)$. Now, for any $x\in\rn$, let
\begin{equation}\label{1018-x}
M_\triangledown(f)(x):=\sup_{t\in(0,\fz),\,|y-x|\le
3(|x_0|+1)t}[|F(y,t)|+|G(y,t)|].
\end{equation}
Then $M_\triangledown(f)$ is lower semi-continuous and,  by
Remark \ref{r-1214} and \cite[Corollary 3.8]{yyyz16}, we further know that
$M_\triangledown(f)\in \wlv$ and, moreover,
\begin{equation}\label{1207y}
\|M_\triangledown(f)\|_{\wlv}\sim \|f\|_{\whv}
\end{equation}
with the implicit equivalent positive constants independent of $f$.

Now we have the following decomposition for elements of the variable weak Hardy space.

\begin{prop}\label{p-czd}
Assume that $\theta\in(0,1)$, $p(\cdot)$ and $\wz p(\cdot)$ are as in Theorem \ref{t-inter}.
Let $f\in W\!H^{\wz p(\cdot)}(\rn)$ and $\az\in(0,\fz)$. Then there exist $g_\az\in L^\fz(\rn)$
and $b_\az\in \cs'(\rn)$ such that $f=g_\az+b_\az$ in $\cs'(\rn)$,
$
\|g_\az\|_{L^\fz(\rn)}\le C_1\az
$
and
\begin{equation}\label{1209u}
\|b_\az\|_{H^{p(\cdot)}(\rn)}
\le C_2\|M_\triangledown(f)\chi_{\{x\in\rn:\ M_\triangledown(f)(x)>\az\}}\|_{\lv}<\fz,
\end{equation}
where $C_1$ and $C_2$ are two positive constants independent of $f$ and $\az$.
\end{prop}

In what follows, to simplify the presentation of this article, 
for any measurable function $g$ and $\az,\,\bz\in(0,\fz)$ with $\az<\bz$,
we \emph{always write} $\chi_{\{x\in\rn:\ g(x)>\az\}}$ or $\chi_{\{x\in\rn:\ \az<g(x)\le\bz\}}$ simply by 
$\chi_{\{g>\az\}}$ or $\chi_{\{\az<g\le\bz\}}$.

\begin{rem}\label{r-decom}
It was established in
\cite[Lemma A]{frs74} that, if $f\in\cs(\rn)$ and $\az\in(0,\fz)$, then $f$ can be written as
the sum of two functions, $g$ and $b$, such that
$\|g\|_{L^\fz(\rn)}\le \wz C_1\az$ and
$$\|b\|_{H^p(\rn)}\le \wz C_2\|M_\triangledown(f)\chi_{\{M_\triangledown(f)>\az\}}\|_{L^p(\rn)},$$
where $C_1$ and $C_2$ are two positive constants independent of $\az$ and $f$.
Comparing with \cite[Lemma A]{frs74}, Proposition \ref{p-czd} presents a decomposition
of any distribution from $W\!H^{\wz p(\cdot)}(\rn)$ and, in this sense, it is a
very useful improvement of \cite[Lemma A]{frs74}.
\end{rem}

To prove Proposition \ref{p-czd}, we need the following vector-valued inequality of the
Hardy-Littlewood maximal operator $\cm$ on the variable Lebesgue space, which was obtained by
Cruz-Uribe and Fiorenza \cite[Corollary 2.1]{cf06}. Here and hereafter, the operator
$\cm$ is defined by setting, for any locally integrable function $f$ on $\rn$ and $x\in\rn$,
$$\cm(f)(x):=\sup_{B\ni x}\frac1{|B|}\int_B|f(y)|\,dy,$$
where the supremum is taken over all balls $B$ of $\rn$ containing $x$.

\begin{lem}\label{l1230-a}
Let $r\in(1,\fz)$. Assume that $p(\cdot)\in C^{\log}(\rn)$ satisfies $1<p_-\le p_+<\fz$.
Then there exists a positive constant $C$ such that, for any sequence $\{f_j\}_{j\in\nn}$
of measurable functions,
$$\lf\|\lf\{\sum_{j\in\nn}[\cm(f_j)]^r\r\}^{1/r}\r\|_{\lv}
\le C\lf\|\lf(\sum_{j\in\nn}|f_j|^r\r)^{1/r}\r\|_{\lv}.$$
\end{lem}

The atomic characterization of $\hv$, obtained by Nakai and Sawano \cite{ns12}, also
plays a key role in the proof of Proposition \ref{p-czd}. The following notions of
$(p(\cdot),\fz)$-atoms and the atomic Hardy space $H_{\rm atom}^{p(\cdot),\fz}(\rn)$
come from \cite[Definition 1.4]{ns12}
and \cite[Definition 1.5]{ns12} of Nakai and Sawano, respectively.

\begin{defn}\label{d1230}
Let $p(\cdot)\in C^{\log}(\rn)$.
\begin{enumerate}
\item[(i)] A measurable function $a$ on $\rn$ is called a \emph{$(p(\cdot),\fz)$-atom} if there exists a
cube $Q\st \rn$ such that $\supp a\st Q$,
$$\|a\|_{L^\fz(\rn)}\le \frac1{\|\chi_Q\|_{\lv}}$$
and $\int_\rn a(x)x^\az\,dx=0$ for any $\az\in \zz_+^n$ with
$|\az|\le \lfloor n(\frac1{p_-}-1)\rfloor$.

\item[(ii)] The \emph{variable atomic Hardy space $H_{\rm atom}^{p(\cdot),\fz}(\rn)$}
is defined as the space of all $f\in \cs'(\rn)$ such that $f=\sum_{j\in\nn}\lz_j a_j$ in
$\cs'(\rn)$, where $\{\lz_j\}_{j\in\nn}$ is a sequence of non-negative numbers and
$\{a_j\}_{j\in\nn}$ is a sequence of $(p(\cdot),\fz)$-atoms, associated with
cubes $\{Q_j\}_{j\in\nn}$ of $\rn$. Moreover, for any $f\in H_{\rm atom}^{p(\cdot),\fz}(\rn)$,
let
$$\|f\|_{H_{\rm atom}^{p(\cdot),\fz}(\rn)}
:=\inf\lf\{\lf\|\lf[\sum_{j\in\nn}\lf\{\frac{\lz_j\chi_{Q_j}}
{\|\chi_{Q_j}\|_{\lv}}\r\}^{\underline{p}}\r]^{1/\underline{p}}\r\|_{\lv}\r\},$$
where $\underline{p}$ is as in \eqref{1219a} and the infimum is taken over all admissible
decompositions of $f$ as above.
\end{enumerate}
\end{defn}

The following atomic characterization of $\hv$ by means of $(p(\cdot),\fz)$-atoms is
a part of \cite[Theorem 4.5]{ns12} of Nakai and Sawano.

\begin{lem}\label{l1230b}
Let $p(\cdot)\in C^{\log}(\rn)$. Then $\hv=H_{\rm atom}^{p(\cdot),\fz}(\rn)$ with
equivalent quasi-norms.
\end{lem}

Now we can show Proposition \ref{p-czd} by using Lemmas \ref{l1230-a} and \ref{l1230b} as follows.

\begin{proof}[Proof of Proposition \ref{p-czd}]
Let $f\in W\!H^{\wz p(\cdot)}(\rn)$ and, for any $i\in\zz$,
$$\boz_i:=\lf\{x\in\rn:\ M_{\triangledown}(f)(x)>2^i\r\}.$$
Then $\boz_i$ is open and, by \eqref{1207y}, we find that
\begin{align*}
\sup_{i\in\zz}2^i\|\chi_{\boz_i}\|_{\lv}
\le\|M_{\triangledown} (f)\|_{\wlv}\sim\|f\|_{\whv}.
\end{align*}
By the Whitney decomposition (see, for example, \cite[p.\,463]{g09-1}),
we know that, for any $i\in\zz$, there exists a sequence $\{\qij\}_{j\in\nn}$ of cubes
such that
\begin{enumerate}
\item[{\rm (i)}] $\bigcup_{j\in\nn} \qij = \boz_i$ and
$\{\qij\}_{j\in\nn}$ have disjoint interiors;

\item[{\rm (ii)}] for any $j\in\nn$, $\sqrt{n}\,l_{\qij}
\le \dist(\qij,\, \boz_i^{\complement}) \le 4 \sqrt{n}\,l_{\qij }$,
where $l_{\qij }$ denotes the side length of the cube $\qij$ and
$\dist(\qij,\, \boz_i^{\complement}):=\inf\{|x-y|:\ x\in\qij,\ y\in \boz_i^{\complement}\}$;

\item[{\rm (iii)}] for any $j,\, k\in\nn$, if the boundaries of
cubes $\qij$ and $Q_{i,k}$ touch, then
$\frac14\le\frac{l_{\qij}}{ l_{Q_{i,k}}}\le 4;$

\item[{\rm (iv)}] for any given $j\in\nn$, there exist at most
$12 n$ different cubes $Q_{i,k}$ that touch $\qij$.
\end{enumerate}
For any $i\in\zz$, $j\in\nn$ and $x\in\rn$, let
$$\dist\lf(x,\boz_i^\com\r):=\inf\lf\{|x-y|:\ y\in \boz_i^\com\r\},$$
$$\wz\boz_i:=\lf\{(x,t)\in\rr^{n+1}_+:\ 0<2t(|x_0|+1)<\dist(x, \boz_i^\com)\r\},$$
$$\wqij:=\lf\{(x,t)\in\rr_+^{n+1}
:\ x\in\qij,\ (x,t)\in\wz\boz_i\bh\wz\boz_{i+1}\r\},$$
and, for any $\epsilon\in(0,\fz)$,
$$\bije(x):=\int_{\epsilon}^{\fz}\int_\rn\chi_{\wqij}(y,t) F(y,t)
\wz \psi _t(x-y)\,dy\frac{dt}{t}$$
and
\begin{equation*}
b_i^{\ez}(x):=\int_\ez^\fz\int_\rn\chi_{\Omega_i^\ast}(y,t)F(y,t)
\wz \psi _t(x-y)\,dy\frac{dt}t,
\end{equation*}
where $$\Omega_i^\ast:=\{(y,t)\in\rr_+^{n+1}:\
y\in \Omega_i,\ (y,t)\in\wz\boz_i\bh\wz\boz_{i+1}\}.$$ It is easy to see that, for any $x\in\rn$,
$\{\dist(x,\, \boz_i^{\complement})\}_{i\in\zz}$
is decreasing in $i$, since $\{\Omega_i\}_{i\in\zz}$ is decreasing in $i$.

Next we finish the proof of this proposition by fours steps.

\textbf{Step 1)} In this step, we show that, for any $\ez\in(0,\fz)$ and $i\in\zz$,
\begin{equation}\label{1214b}
\|b_i^\ez\|_{L^\fz(\rn)}\ls 2^i
\end{equation}
with the implicit positive constant independent of $\ez$, $i$ and $f$.

To this end, we first observe that, for any given $(x,t)\in \rr_+^{n+1}$,
$$\supp({\wz\psi_t(x-\cdot)})\subset B(x,t[|x_0|+1]).$$
For any $i\in\zz$, let
$\Omega_i^{\ast,1}:=\{(y,t)\in\rr_+^{n+1}:\ y\in \Omega_i,\ (y,t)\in\wz\boz_i\}$
and
$$\Omega_i^{\ast,2}:=\lf\{(y,t)
\in\rr_+^{n+1}:\ y\in \Omega_i,\ (y,t)\in(\wz\boz_{i+1})^\complement\r\}.$$
Then $\chi_{\Omega_i^\ast}=\chi_{\Omega_i^{\ast,1}}\chi_{\Omega_i^{\ast,2}}$.

If $t\ge \frac{\dist(x,\Omega_i^\complement)}{|x_0|+1}$, then, for any $(y,t)\in \rr_+^{n+1}$
with $y\in B(x,t[|x_0|+1])$,
\begin{align*}
\dist\lf(y,\Omega_i^\complement\r)\le
\dist\lf(x,\Omega_i^\complement\r)+|x-y|
< \dist\lf(x,\Omega_i^\complement\r)+t(|x_0|+1)\le 2t(|x_0|+1),
\end{align*}
which implies that $(y,t)\notin \wz\Omega_i$ and hence
$\chi_{\Omega_i^{\ast,1}}(y,t)\wz \psi _t(x-y)=0$.

If $t< \frac{\dist(x,\Omega_i^\complement)}{3(|x_0|+1)}$, then, for any $(y,t)\in \rr_+^{n+1}$
with $y\in B(x,t[|x_0|+1])$,
\begin{align}\label{1230a}
\dist\lf(y,\Omega_i^\complement\r)
&\ge
\dist\lf(x,\Omega_i^\complement\r)-|x-y|> \dist\lf(x,\Omega_i^\complement\r)-t(|x_0|+1)\\
&>2t(|x_0|+1),\noz
\end{align}
which implies that $(y,t)\in \wz\Omega_i$, $y\in\Omega_i$ and hence
$\chi_{\Omega_i^{\ast,1}}(y,t)\wz \psi _t(x-y)=\wz \psi _t(x-y)$.

If $t\ge \frac{\dist(x,\Omega_{i+1}^\complement)}{|x_0|+1}$, then, for any $(y,t)\in \rr_+^{n+1}$
with $y\in B(x,t[|x_0|+1])\cap \Omega_i$,
\begin{align}\label{1230b}
\dist\lf(y,\Omega_{i+1}^\complement\r)
&\le
\dist\lf(x,\Omega_{i+1}^\complement\r)+|x-y|
<\dist\lf(x,\Omega_{i+1}^\complement\r)+t(|x_0|+1)\\
&\le 2t(|x_0|+1),\noz
\end{align}
which implies that $(y,t)\notin \wz\Omega_{i+1}$ and hence
$\chi_{\Omega_{i+1}^{\ast,2}}(y,t)\wz \psi _t(x-y)=\wz \psi _t(x-y)$.

If $t< \frac{\dist(x,\Omega_{i+1}^\complement)}{3(|x_0|+1)}$, then, for any $(y,t)\in \rr_+^{n+1}$
with $y\in B(x,t[|x_0|+1])$,
\begin{align*}
\dist\lf(y,\Omega_{i+1}^\complement\r)\ge
\dist\lf(x,\Omega_{i+1}^\complement\r)-|x-y|> \dist\lf(x,\Omega_{i+1}^\complement\r)-t(|x_0|+1)
>2t(|x_0|+1),
\end{align*}
which implies that $(y,t)\in \wz\Omega_{i+1}$ and hence
$\chi_{\Omega_{i+1}^{\ast,1}}(y,t)\wz \psi _t(x-y)=0$.

From the above arguments, we deduce that, for any given $(x,t)\in\rr_+^{n+1}$, if
$$t\notin \lf[\frac{\dist(x,\Omega_{i+1}^\complement)}{3(|x_0|+1)},
\frac{\dist(x,\Omega_{i}^\complement)}{|x_0|+1}\r],$$ then, for any $y\in\rn$, it holds true that
\begin{equation}\label{1214a}\chi_{\Omega_{i}^{\ast}}(y,t)\wz \psi _t(x-y)
=\chi_{\Omega_{i}^{\ast,1}}(y,t)\chi_{\Omega_{i}^{\ast,2}}(y,t)\wz \psi _t(x-y)=0.
\end{equation}

Next, for any given $i\in\zz$ and $x\in\rn$, we estimate $|b_i^\ez(x)|$ in two cases.

\emph{Case 1.1)} $\frac{\dist(x,\Omega_i^\complement)}{3(|x_0|+1)}
\le\frac{\dist(x,\Omega_{i+1}^\complement)}{|x_0|+1}$.

In this case, we have
\begin{equation*}
\frac{\dist (x,\Omega_{i+1}^\complement)}{3(|x_0|+1)}
\le \frac{\dist(x,\Omega_i^\com)}{3(|x_0|+1)}\le \frac{\dist(x,\Omega_{i+1}^\com)}{|x_0|+1}
\le \frac{\dist(x,\Omega_i^\com)}{|x_0|+1}.
\end{equation*}
Then, by \eqref{1214a}, we know that
\begin{align}\label{1019-x}
|b_i^{\ez}(x)|&\le\int_{\frac{\dist(x,\Omega_{i+1}^\complement)}{3(|x_0|+1)}
}^{\frac{\dist(x,\Omega_{i+1}^\complement)}{|x_0|+1}}
\int_\rn\chi_{\Omega_i^\ast}(y,t)|F(y,t)|
|\wz \psi _t(x-y)|\,dy\frac{dt}t
+\int_{\frac{\dist(x,\Omega_{i}^\complement)}{3(|x_0|+1)}}
^{\frac{\dist(x,\Omega_{i}^\complement)}{|x_0|+1}}\cdots.
\end{align}
Since $\chi_{\Omega_i^\ast}(y,t)\neq 0$ implies that $(y,t)\notin \wz\Omega_{i+1}$, it follows
that $|F(y,t)|\le 2^{i+1}$, where $F$ is as in \eqref{fandg}.
By this, \eqref{1019-x} and the fact that, for any $t\in(0,\fz)$
and $x\in\rn$,
$\int_\rn |\wz \psi _t(x-y)|\,dy\ls1$, we conclude that
\begin{align}\label{1019-y}
|b_i^\ez(x)|\ls 2^i\lf\{\int_{\frac{\dist(x,\Omega_{i+1}^\complement)}{3(|x_0|+1)}}
^{\frac{\dist(x,\Omega_{i+1}^\complement)}{|x_0|+1}}\,\frac{dt}t
+\int_{\frac{\dist(x,\Omega_{i}^\complement)}{3(|x_0|+1)}}
^{\frac{\dist(x,\Omega_{i}^\complement)}{|x_0|+1}}\,\frac{dt}t\r\}\ls 2^i
\end{align}
with the implicit positive constants independent of $\ez$, $i$ and $x$.

\emph{Case 1.2)} $\frac{\dist(x,\Omega_i^\complement)}{3(|x_0|+1)}
>\frac{\dist(x,\Omega_{i+1}^\complement)}{|x_0|+1}$.

In this case, we have
$$ \frac{\dist(x,\Omega_{i+1}^\com)}{|x_0|+1}<\frac{\dist(x,\Omega_i^\com)}{3(|x_0|+1)}
<\frac{\dist(x,\Omega_i^\com)}{|x_0|+1}.$$
Then, by \eqref{1230a}, \eqref{1230b} and \eqref{1214a}, we obtain
\begin{align*}
\lf|b_i^\ez(x)\r|&\le\int_{\frac{\dist(x,\Omega_{i+1}^\complement)}{3(|x_0|+1)}}
^{\frac{\dist(x,\Omega_{i+1}^\complement)}{|x_0|+1}}
\int_\rn\chi_{\Omega_i^\ast}(y,t)|F(y,t)|
|\wz \psi _t(x-y)|\,dy\frac{dt}t
+\int_{\frac{\dist(x,\Omega_{i}^\complement)}{3(|x_0|+1)}}
^{\frac{\dist(x,\Omega_{i}^\complement)}{|x_0|+1}}\cdots\\
&\hs\hs+\lf|\int_{\frac{\dist(x,\Omega_{i+1}^\complement)}{|x_0|+1}}
^{\frac{\dist(x,\Omega_{i}^\complement)}{3(|x_0|+1)}}
\int_\rn F(y,t)
\wz \psi _t(x-y)\,dy\frac{dt}t\r|\\
&=:{\rm I}_1+{\rm I}_2+{\rm I}_3.
\end{align*}
By the argument same as that used in the proof of \eqref{1019-y}, we have
$${\rm I}_1+{\rm I}_2\ls 2^i.$$

For ${\rm I}_3$, by \eqref{1019-z},
we find that
\begin{equation}\label{1019-u}
{\rm I}_3=\lf|G\lf(x,\frac{\dist(x,\Omega_{i+1}^\complement)}{|x_0|+1}\r)
-G\lf(x,\frac{\dist(x,\Omega_{i}^\complement)}{3(|x_0|+1)}\r)\r|.
\end{equation}
On the other hand, if $t\in[\frac{\dist(x,\Omega_{i+1}^\complement)}{|x_0|+1},
\frac{\dist(x,\Omega_{i}^\complement)}{3(|x_0|+1)}]$, then
$\dist(x,\Omega_{i+1}^\complement)\le t(|x_0|+1)$. Thus, there exists
$x^\ast\in \Omega_{i+1}^\com$ such that $|x-x^\ast|\le 3t(|x_0|+1)$, which, together with
\eqref{1018-x}, further implies that
$$|G(x,t)|\ls M_\triangledown(f)(x^\ast)\ls 2^{i}.$$
By this and \eqref{1019-u}, we find that ${\rm I}_3\le 2^i$. Therefore, in this case,
we also have $|b_i^\ez(x)|\ls 2^i$, which, combined with \eqref{1019-y}, implies that
\eqref{1214b} holds true. This finishes the proof of Step 1).

\textbf{Step 2)} In this step, we construct $g_\az$ for any $\az\in(0,\fz)$.

By \eqref{1214b} of Step 1), we conclude that,
for any $i\in\zz$, $\{b_i^\ez\}_{\ez\in(0,\fz)}$ is bounded in $L^\fz(\rn)$ uniformly with respect
to $\ez\in(0,\fz)$. Thus, by the Alaoglu theorem (see, for example, \cite[Theorem 3.17]{Ru91})
and the well-known diagonal rule, we find that
there exist $\{b_i\}_{i\in\zz}\st L^\fz(\rn)$ and $\{\ez_k\}_{k\in\nn}\st(0,\fz)$ such that
$\ez_k\to0$ as $k\to\fz$ and, for any $i\in\zz$ and $g\in L^1(\rn)$,
\begin{equation}\label{1019-v}
\lim_{k\to\fz}\langle b_i^{\ez_k},g\rangle=\langle b_i, g\rangle;
\end{equation}
moreover, for any $i\in\zz$,
\begin{equation}\label{1209x}
\|b_i\|_{L^\fz(\rn)}\ls 2^i.
\end{equation}

Next, we claim that, for any $i_0\in\zz$,
\begin{equation}\label{1019-w}
\lim_{k\to\fz}\sum_{i=-\fz}^{i_0}b_i^{\ez_k}=\sum_{i=-\fz}^{i_0}b_i\quad{\rm in}\quad \cs'(\rn).
\end{equation}
Indeed, by \eqref{1214b}, we know that, for any $\ez\in(0,\fz)$,
\begin{align}\label{1209y}
\sum_{i=-\fz}^{i_0}\|b_i^\ez\|_{L^\fz(\rn)}\ls \sum_{i=-\fz}^{i_0}2^i\ls 2^{i_0},
\end{align}
and, by \eqref{1209x}, we have
\begin{equation}\label{1209z}
\sum_{i=-\fz}^{i_0}\|b_i\|_{L^\fz(\rn)}\ls \sum_{i=-\fz}^{i_0}2^i\ls 2^{i_0}.
\end{equation}
Thus, both $\sum_{i=-\fz}^{i_0}b_i^\ez$ and $\sum_{i=-\fz}^{i_0}b_i$ converge in $L^\fz(\rn)$
and, for any $g\in \cs(\rn)$,
\begin{align*}
\sum_{i=-\fz}^{i_0}\lf[|\langle b_i^{\ez_k},g\rangle|+|\langle b_i,g\rangle|\r]
&\le \|g\|_{L^1(\rn)}\sum_{i=-\fz}^{i_0}[\|b_i^\ez\|_{L^\fz(\rn)}+\|b_i\|_{L^\fz(\rn)}]\\
&\ls 2^{i_0}\|g\|_{L^1(\rn)}<\fz.
\end{align*}
Therefore, for any $\delta\in(0,\fz)$, there exists $N\in\nn\cap[|i_0|,\fz)$, depending on $\delta$ and $n$, such that
$$\sum_{i=-\fz}^{-N}\lf[|\langle b_i^{\ez_k},g\rangle|+|\langle b_i,g\rangle|\r]<\delta\|g\|_{L^1(\rn)}/4.$$

On the other hand, by \eqref{1019-v}, we find that, for any given $i\in\zz$ and $g\in\cs(\rn)$,
there exists $K_i\in\nn$,  depending on $i$ and $g$, such that,
when $k>K_i$,
$$|\langle b_i^{\ez_k},g\rangle- \langle b_i,g\rangle|<\frac{\delta}{2(|i_0|+N)}.$$
Let $K:=\max\{K_{-N+1}, K_{-N+2},\dots,K_{i_0}\}$. Then,  for any given $i\in\zz$ and $g\in\cs(\rn)$, 
by \eqref{1209y} and \eqref{1209z},
we know that, for any $k>K$,
\begin{align*}
&\lf|\lf\langle \sum_{i=-\fz}^{i_0}b_i^{\ez_k}, g\r\rangle
-\lf\langle  \sum_{i=-\fz}^{i_0}b_i, g\r\rangle\r|\\
&\hs =\lf|\sum_{i=-\fz}^{i_0}\lf\langle b_i^{\ez_k}, g\r\rangle
-\sum_{i=-\fz}^{i_0}\lf\langle  b_i, g\r\rangle\r|\\
&\hs \le \sum_{i=-\fz}^{-N}\lf[|\langle b_i^{\ez_k},g\rangle|+|\langle b_i,g\rangle|\r]
+\sum_{i=-N+1}^{i_0}\lf|\lf\langle b_i^{\ez_k}, g\r\rangle-\lf\langle  b_i, g\r\rangle\r|\\
&\hs \le \delta\|g\|_{L^1(\rn)}/4+(|i_0|+N)\frac{\delta}{2(|i_0|+N)}\ls\delta,
\end{align*}
where the implicit positive constant depends on $i$ and $g$.
This implies that \eqref{1019-w} holds true.

Now, for any given $\az\in(0,\fz)$, we choose $i_0\in\zz$ such that
$2^{i_0}\le\az<2^{i_0+1}$. Let
$$g_\az:=\sum_{i=-\fz}^{i_0} b_i.$$
Then, by the above claim and \eqref{1209z}, we conclude that $g_\az\in L^\fz(\rn)$
and
$$\|g_\az\|_{L^\fz(\rn)}\ls 2^{i_0}\sim \az,$$
which completes the proof of Step 2).

\textbf{Step 3)} In this step, we construct $b_\az$ for any $\az\in(0,\fz)$.

From an argument similar to that used in the proof
of \cite[Theorem 4.4]{yyyz16}, we deduce that there exist
$\{b_{i,j}\}_{i>i_0,j\in\nn}\st L^\fz(\rn)$ and a subsequence
$\{\ez_{k_l}\}_{l\in\nn}\st \{\ez_k\}_{k\in\nn}$ such that
$\ez_{k_l}\to 0$ as $l\to\fz$ and, for any $i>i_0$, $j\in\nn$ and $g\in L^1(\rn)$,
$$\lim_{l\to\fz}\langle b_{i,j}^{\ez_{k_l}},g\rangle=\langle b_{i,j},g\rangle,$$
$\supp b_{i,j}\st c_1Q_{i,j}$ for some constant $c_1\in(1,\fz)$,
$\|b_{i,j}\|_{L^\fz(\rn)}\ls 2^i$
and $\int_\rn b_{i,j}(x)x^\gamma\,dx=0$ for any $\gamma\in\zz_+^n$ with
$|\gamma|\le s_0$, where $s_0$ is as in \eqref{1218a}.
Moreover,
$$\lim_{l\to\fz}\sum_{i=i_0+1}^\fz\sum_{j\in\nn}b_{i,j}^{\ez_{k_l}}
=\sum_{i=i_0+1}^\fz\sum_{j\in\nn}b_{i,j}\quad{\rm in}\quad \cs'(\rn).$$

Let
$$b_\az :=\sum_{i=i_0+1}^\fz\sum_{j\in\nn}b_{i,j}.$$
Next, we show that $b_\az$ satisfies \eqref{1209u}.
By the construction of $b_{i,j}$, we know that there exists $\wz c_1\in(0,\fz)$ such that
\begin{equation}\label{1209-a}
b_\az=\sum_{i=i_0+1}^\fz\sum_{j\in\nn}\lf[\wz c_12^i\|\chi_{c_1Q_{i,j}}\|_{\lv}\r]
\frac{b_{i,j}}{\wz c_12^i\|\chi_{c_1Q_{i,j}}\|_\lv}
=:\sum_{i=i_0+1}^\fz \lz_{i,j}\wz b_{i,j}
\end{equation}
is a linear combination of $(p(\cdot),\fz)$-atoms (see Definition \ref{d1230}(i)).
Then, by Lemma \ref{l1230-a}, the fact that
$\chi_{c_1Q_{i,j}}\ls \cm(\chi_{Q_{i,j}})$ and property (i) of the
previous Whitney decomposition presented in this proof, we find that
\begin{align}\label{1209w}
&\lf\|\lf\{\sum_{i=i_0+1}^\fz\sum_{j\in\nn}
\lf[\frac{\lz_{i,j}\chi_{c_1Q_{i,j}}}{\|\chi_{c_1Q_{i,j}}\|_\lv}\r]^{\underline{p}}
\r\}^{\frac1{\underline{p}}}\r\|_{\lv}\\
&\hs \ls\lf\|\lf\{\sum_{i=i_0+1}^\fz\sum_{j\in\nn}2^{i\underline{p}}
\cm(\chi_{Q_{i,j}})\r\}^{\frac1{\underline{p}}}\r\|_{\lv}
\ls \lf\|\lf\{\sum_{i=i_0+1}^\fz\sum_{j\in\nn}2^{i\underline{p}}
\chi_{Q_{i,j}}\r\}^{\frac1{\underline{p}}}\r\|_{\lv}\noz\\
&\hs\ls \lf\|\lf\{\sum_{i=i_0+1}^\fz2^{i\underline{p}}\chi_{\Omega_i}\r\}^{\frac1{\underline{p}}}
\r\|_{\lv}
\sim \lf\|\lf\{\sum_{i=i_0+1}^\fz[2^{i}\chi_{\Omega_i\backslash \Omega_{i+1}}]^{\underline{p}}
\r\}^{\frac1{\underline{p}}}\r\|_{\lv}\noz\\
&\hs\ls \lf\|M_{\triangledown}(f)\lf\{\sum_{i=i_0+1}^\fz\chi_{\Omega_i\backslash \Omega_{i+1}}
\r\}^{\frac1{\underline{p}}}\r\|_{\lv}
\sim \lf\|M_{\triangledown}(f)\chi_{\Omega_{i_0+1}}\r\|_{\lv}\noz\\
&\hs\ls \lf\|M_{\triangledown}(f)\chi_{\{M_{\triangledown}(f)>\az\}}\r\|_{\lv},\noz
\end{align}
where $\underline{p}$ is as in \eqref{1219a}.
On the other hand, by the definition of $W\!L^{\wz p(\cdot)}(\rn)$
and \eqref{1207y}, we know that, for $\theta\in(0,1)$,
\begin{align}\label{1209v}
&\lf\|M_{\triangledown}(f)\chi_{\{M_{\triangledown}(f)>\az\}}\r\|_{\lv}\\
&\hs\le\lf\{\sum_{i=0}^\fz\lf\|[M_\triangledown(f)]^{\underline{p}}
\chi_{\{2^i\az<M_\triangledown(f)\le 2^{i+1}\az\}}\r\|_{L^{\frac{p(\cdot)}{\underline{p}}}(\rn)}
\r\}^{\frac1{\underline{p}}}\noz\\
&\hs\sim \lf\{\sum_{i=0}^\fz(2^i\az)^{\underline{p}}\|\chi_{\{M_\triangledown(f)>2^i\az\}}
\|_{\lv}^{\underline{p}}
\r\}^{\frac1{\underline{p}}}\noz\\
&\hs\sim \lf\{\sum_{i=0}^\fz(2^i\az)^{-\theta\underline{p}/(1-\theta)}
\lf[2^i\az \|\chi_{\{M_\triangledown(f)>2^i\az\}}
\|_{L^{\wz p(\cdot)}(\rn)}\r]^{\underline{p}/(1-\theta)}
\r\}^{\frac1{\underline{p}}}\noz\\
&\hs\ls \|M_\triangledown(f)\|_{W\!H^{\wz p(\cdot)}(\rn)}^{\frac1{1-\theta}}
\lf\{\sum_{i=0}^\fz(2^i\az)^{-\theta\underline{p}/(1-\theta)}\r\}^{\frac1{\underline{p}}}\noz\\
&\hs\sim \|M_\triangledown(f)\|_{W\!L^{\wz p(\cdot)}(\rn)}^{\frac1{1-\theta}} \az^{-\theta/(1-\theta)}
\sim \|f\|_{W\!H^{\wz p(\cdot)}(\rn)}^{\frac1{1-\theta}}\az^{-\theta/(1-\theta)}<\fz.\noz
\end{align}

Combining \eqref{1209-a}, \eqref{1209w} and \eqref{1209v}, we conclude that
$b_\az$ belongs to the variable atomic Hardy space $H_{\rm atom}^{p(\cdot),\fz}(\rn)$
(see Definition \ref{d1230}(ii)).
By this and Lemma \ref{l1230b}, we find that
\begin{align*}
\|b_\az\|_{\hv}&\sim\|b_\az\|_{H_{\rm atom}^{p(\cdot),\fz}(\rn)}
\ls \lf\|\lf\{\sum_{i=i_0+1}^\fz\sum_{j\in\nn}
\lf[\frac{\lz_{i,j}\chi_{c_1Q_{i,j}}}{\|\chi_{c_1Q_{i,j}}\|_\lv}\r]^{\underline{p}}
\r\}^{\frac1{\underline{p}}}\r\|_{\lv}\\
&\ls\lf\|M_{\triangledown}(f)\chi_{\{M_{\triangledown}(f)>\az\}}\r\|_{\lv}.
\end{align*}
This finishes the proof of Step 3).

\textbf{Step 4)} In this step, we prove that $f=g_\az+b_\az$ in $\cs'(\rn)$.

We first claim that, for any given $i\in\zz$, $\ez\in(0,\fz)$ and $x\in\rn$,
\begin{equation}\label{1129x}
\sum_{j\in\nn}b_{i,j}^\ez(x)=b_i^\ez(x).
\end{equation}
Indeed, since, for any $(y,t)\in\rn\times(0,\fz)$ and all $x\in B(y,3[|x_0|+1]t)$,
$F(y,t)\le M_\triangledown(f)(x)$, it follows that
\begin{align*}
|F(y,t)|\le \|\chi_{B(y,3[|x_0|+1]t)}\|_{\lv}^{-1}
\|M_\triangledown(f)\|_{L^{p(\cdot)}(B(y,3[|x_0|+1]t))},
\end{align*}
which, together with the fact that, for any measurable subset $E\st \rn$,
$$\|\chi_E\|_{\lv}\ge \min\lf\{|E|^{\frac1{p_-}},|E|^{\frac1{p_+}}\r\},$$
implies that
\begin{align}\label{1129y}
|F(y,t)|\le \max\{t^{-\frac{n}{p_-}},t^{-\frac{n}{p_+}}\}
\|M_\triangledown(f)\|_{L^{p(\cdot)}(\rn)}.
\end{align}
Thus, by \eqref{1129y}, $\sum_{j\in\nn}\chi_{\wz Q_{i,j}}=\chi_{\Omega_i^\ast}\le 1$
and the Lebesgue dominated convergence theorem, we conclude that
\begin{align*}
\sum_{j\in\nn}b_{i,j}^\ez(x)
&=\sum_{j\in\nn}\int_{\epsilon}^{\fz}\int_\rn\chi_{\wqij}(y,t) F(y,t)
\wz \psi _t(x-y)\,dy\frac{dt}{t}\\
&=\lim_{N\to\fz}\int_{\epsilon}^{\fz}\int_\rn
\sum_{j=1}^N\chi_{\wqij}(y,t) F(y,t)
\wz \psi _t(x-y)\,dy\frac{dt}{t}\\
&=\int_{\epsilon}^{\fz}\int_\rn\chi_{\Omega_i^\ast}(y,t) F(y,t)
\wz \psi _t(x-y)\,dy\frac{dt}{t}
=b_i^\ez(x).
\end{align*}
Therefore, \eqref{1129x} holds true.

Now, by the argument same as that used in \cite[p.\,2855]{yyyz16},
 \eqref{1019-w} and \eqref{1129x}, we find that
\begin{align*}
f&=\lim_{l\to\fz}\sum_{i\in\zz}\sum_{j\in\nn}b_{i,j}^{\ez_{k_l}}
=\lim_{l\to\fz}\sum_{i=-\fz}^{i_0}\sum_{j\in\nn}b_{i,j}^{\ez_{k_l}}+
\lim_{l\to\fz}\sum_{i=i_0+1}^{\fz}\sum_{j\in\nn}b_{i,j}^{\ez_{k_l}}\\
&=\lim_{l\to\fz}\sum_{i=-\fz}^{i_0}b_{i}^{\ez_{k_l}}+
\lim_{l\to\fz}\sum_{i=i_0+1}^{\fz}\sum_{j\in\nn}b_{i,j}^{\ez_{k_l}}
=\sum_{i=-\fz}^{i_0}b_{i}+
\sum_{i=i_0+1}^{\fz}\sum_{j\in\nn}b_{i,j}
=g_\az+b_\az,
\end{align*}
where all above summations converge in $\cs'(\rn)$.
This finishes the proof of Step 4) and hence the proof of Proposition \ref{p-czd}.
\end{proof}

\section{Proofs of Theorem \ref{t-inter} and Corollary \ref{cor}\label{s-2}}
\hskip\parindent
In this section, we prove Theorem \ref{t-inter} and Corollary \ref{cor} by using Proposition
\ref{p-czd}. To this end, we need the following known real interpolation result, which is
\cite[Theorem 4.1]{kv14} of Kempka and Vyb\'iral in the case when $p(\cdot)=q_0(\cdot)$ and $q=\fz$.

\begin{lem}\label{l1230c}
Let $p(\cdot)\in\cp(\rn)$ and $\theta\in(0,1)$.
Then it holds true that
$$(\lv,L^\fz(\rn))_{\theta,\fz}=W\!L^{p(\cdot)/(1-\theta)}(\rn).$$
\end{lem}

Using Proposition \ref{p-czd} and Lemma \ref{l1230c}, we now show Theorem \ref{t-inter} as follows.

\begin{proof}[Proof of Theorem \ref{t-inter}]
We first prove that
\begin{equation}\label{1218b}
W\!H^{\wz p(\cdot)}(\rn)\st (H^{p(\cdot)}(\rn),L^\fz(\rn))_{\theta,\fz}.
\end{equation}
Let $f\in W\!H^{\wz p(\cdot)}(\rn)$. Then we show
$f\in (H^{p(\cdot)}(\rn),L^\fz(\rn))_{\theta,\fz}$ by two steps.

\textbf{Step 1)} In this step, we estimate $K(t,f;H^{p(\cdot)}(\rn),L^\fz(\rn))$ for any $t\in(0,\fz)$.

By Proposition \ref{p-czd},
we know that, for any $\az\in(0,\fz)$,
there exist $g_\az\in L^\fz(\rn)$ and $b_\az\in H^{p(\cdot)}(\rn)$
such that $f= g_\az+b_\az$ in $\cs'(\rn)$, $\|g_\az\|_{L^\fz(\rn)}\ls \az$ and
$b_\az$ satisfies \eqref{1209u}. Then, by Proposition \ref{p-czd}, we find that, for any $t\in(0,\fz)$,
\begin{align*}
&K(t,f;H^{p(\cdot)}(\rn),L^\fz(\rn))\\
&\hs=\inf\lf\{\|f_0\|_{\hv}+t\|f_1\|_{L^\fz(\rn)}:\
f=f_0+f_1,\ f_0\in\hv\ {\rm and}\ f_1\in L^\fz(\rn)\r\}\\
&\hs\le\inf_{\az\in(0,\fz)}\lf\{\|b_\az\|_{\hv}+t\|g_\az\|_{L^\fz(\rn)}:
\ b_\az\ {\rm and}\ g_\az\ {\rm are
\ as\ in\  Proposition\ \ref{p-czd}}\r\}\\
&\hs\ls\inf_{\az\in(0,\fz)}\lf\{\lf\|M_\triangledown (f)\chi_{\{M_\triangledown (f)>\az\}}
\r\|_{\lv}+t\az\r\}\\
&\hs\sim \inf_{\az\in(0,\fz)}\lf\{\lf\|M_\triangledown (f)\sum_{j=0}^\fz
\chi_{\{2^j\az<M_\triangledown (f)\le 2^{j+1}\az\}}\r\|_{\lv}+t\az\r\},
\end{align*}
which, together with Remark \ref{r-lv}(i), implies that
\begin{align}\label{1219c}
&K(t,f;H^{p(\cdot)}(\rn),L^\fz(\rn))\\
&\hs\ls \inf_{\az\in(0,\fz)}\lf\{\lf[\sum_{j=0}^\fz\lf\|M_\triangledown (f)
\chi_{\{2^j\az<M_\triangledown (f)\le 2^{j+1}\az\}}\r\|_{\lv}^{\underline{p}}
\r]^{\frac1{\underline{p}}}+t\az\r\}\noz\\
&\hs\ls \inf_{\az\in(0,\fz)}\lf\{\lf[\sum_{j=0}^\fz(2^j\az)^{\underline{p}}\lf\|
\chi_{\{M_\triangledown (f)>2^j\az\}}\r\|_{\lv}^{\underline{p}}
\r]^{\frac1{\underline{p}}}+t\az\r\}\noz\\
&\hs\ls \inf_{\az\in(0,\fz)}\lf\{\lf[\sum_{j=0}^\fz[2^j\az h(2^j\az)]^{\underline{p}}
\r]^{\frac1{\underline{p}}}+t\az\r\},\noz
\end{align}
where, for any $\lz\in(0,\fz)$,
$$h(\lz):=\|\chi_{\{M_{\triangledown}(f)>\lz\}}\|_{\lv}.$$

For any $t\in(0,\fz)$, let
$$\az:=\az(t):=\inf\lf\{\mu\in(0,\fz):\ \lf[\sum_{j=0}^\fz\lf\{2^j h(2^j\mu)\r\}^{\underline{p}}
\r]^{\frac1{\underline{p}}}\le t\r\}.$$
Observe that the function $h$ is decreasing on $(0,\fz)$.
Then it is easy to see that
$$\lf[\sum_{j=0}^\fz\lf\{2^j h(2^j\az(t))\r\}^{\underline{p}}
\r]^{\frac1{\underline{p}}}\le t.$$
From this and \eqref{1219c}, we deduce that
\begin{align}\label{1206x}
K(t,f;H^{p(\cdot)}(\rn),L^\fz(\rn))\ls t\az(t).
\end{align}

\textbf{Step 2)} In this step, we estimate
$${\rm I}:=\sup_{t\in(0,\fz)}t^{-\theta}
K(t,f;H^{p(\cdot)}(\rn),L^\fz(\rn)).$$

By \eqref{1206x}, we find that
\begin{align}\label{1219d}
{\rm I}\ls\sup_{t\in(0,\fz)}t^{1-\theta}\az(t)
\sim\sup_{k\in\zz}\sup_{\gfz{t\in(0,\fz)}{2^k<\az(t)\le 2^{k+1}}}
t^{1-\theta}2^k.
\end{align}
Notice that, when $2^k<\az(t)$,
$$\lf\{\sum_{j=0}^\fz[2^jh(2^j2^k)]^{\underline{p}}\r\}^{\frac1{\underline{p}}}\ge t,$$
which, combined with \eqref{1219d}, implies that
\begin{align*}
{\rm I}\ls \sup_{k\in\zz}2^k\lf\{\sum_{j=0}^\fz[2^jh(2^j2^k)]^{\underline{p}}
\r\}^{\frac{1-\theta}{\underline{p}}}.
\end{align*}

If $\frac{1-\theta}{\underline{p}}\le 1$ with $\underline{p}$ as in \eqref{1219a},
then, by the well-known inequality that,
for any $d\in(0,1]$ and $\{a_i\}_{i\in\nn}\st\cc$,
$$\lf(\sum_{i\in\nn}|a_i|\r)^d\le \sum_{i\in\nn}|a_i|^d,$$
we know that
\begin{align}\label{1206y}
{\rm I}&\ls \sup_{k\in\zz}2^k\sum_{j=0}^\fz[2^jh(2^{j+k})]^{1-\theta}
\ls \sup_{k\in\zz}\sum_{j=0}^\fz 2^{-\theta j}2^{j+k}[h(2^{j+k})]^{1-\theta}\\
&\ls\sum_{j=0}^\fz2^{-\theta j}\sup_{l\in\zz}2^l[h(2^l)]^{1-\theta}
\ls\sup_{l\in\zz}2^l[h(2^l)]^{1-\theta}
\sim\sup_{l\in\zz}2^l\|\chi_{\{M_\triangledown (f)>2^l\}}\|_{\lv}^{1-\theta}\noz\\
&\sim \sup_{l\in\zz}2^l\|\chi_{\{M_\triangledown (f)>2^l\}}\|_{L^{\wz p(\cdot)}(\rn)}
\ls \|M_\triangledown(f)\|_{W\!L^{\wz p(\cdot)}(\rn)}.\noz
\end{align}

If $\frac{1-\theta}{\underline{p}}> 1$ with $\underline{p}$ as in \eqref{1219a},
then, by the H\"older inequality, we find that,
for any $\ez\in(0,\frac{\theta}{1-\theta})$,
\begin{align}\label{1206z}
{\rm I}&\ls\sup_{k\in\zz}2^k\sum_{j=0}^\fz 2^{j(1+\ez)(1-\theta)}[h(2^{j+k})]^{1-\theta}
\ls \sup_{l\in\zz}2^l[h(2^l)]^{1-\theta}\\
&\ls \sup_{l\in\zz}2^l\|\chi_{\{M_\triangledown (f)>2^l\}}\|_{L^{\wz p(\cdot)}(\rn)}
\ls \|M_\triangledown(f)\|_{W\!L^{\wz p(\cdot)}(\rn)}.\noz
\end{align}

Now, by \eqref{1206y} and \eqref{1206z}, we conclude that
\begin{align*}
\sup_{t\in(0,\fz)}t^{-\theta}
K(t,f;H^{p(\cdot)}(\rn),L^\fz(\rn))
\ls \|M_\triangledown(f)\|_{W\!L^{\wz p(\cdot)}(\rn)}
\sim \|f\|_{W\!H^{\wz p(\cdot)}(\rn)},
\end{align*}
which completes the proof of Step 2).
Therefore, $f\in (H^{p(\cdot)}(\rn),L^\fz(\rn))_{\theta,\fz}$ and hence
\eqref{1218b} holds true.

Conversely, we need to show that
\begin{equation}\label{1219e}
(H^{p(\cdot)}(\rn),L^\fz(\rn))_{\theta,\fz}\st W\!H^{\wz p(\cdot)}(\rn).
\end{equation}
To prove \eqref{1219e}, let $T$ be a sublinear operator defined by setting,
for any $f\in\cs'(\rn)$,
$T(f):=f_{N,+}^\ast$,
where $N$ is as in Definition \ref{d-hardy} and $f_{N,+}^\ast$ is as in \eqref{2.8x}.

We claim that the operator $T$ is bounded from the space $(\hv,L^\fz(\rn))_{\theta,\fz}$ to
the space $(\lv,L^\fz(\rn))_{\theta,\fz}$.
Indeed, let $g\in (\hv,L^\fz(\rn))_{\theta,\fz}$. Then, by the definition
of $(\hv,L^\fz(\rn))_{\theta,\fz}$, we know that there exist
$g_0\in \hv$ and $g_1\in L^\fz(\rn)$ such that
\begin{equation}\label{1207x}
\sup_{t\in(0,\fz)}t^{-\theta}[\|g_0\|_{\hv}+t\|g_1\|_{L^\fz(\rn)}]
\ls \|g\|_{(\hv,L^\fz(\rn))_{\theta,\fz}}.
\end{equation}
Moreover, observe that $T(g)\le T(g_0)+T(g_1)$.
Notice that $T$ is bounded from $L^\fz(\rn)$ to $L^\fz(\rn)$ and also from $\hv$ to $\lv$.
It follows that $T(g_0)\in \lv$ and $T(g_1)\in L^\fz(\rn)$.
Let
$$E_0:=\lf\{x\in\rn:\ \frac 12 T(g)(x)\le T(g_0)(x)\r\}\ {\rm and}\
E_1:=\lf\{x\in\rn:\ \frac 12 T(g)(x)\le T(g_1)(x)\r\}.$$
Then $\rn =[E_0\cup E_1]=[E_0\cup(E_1\backslash E_0)]$. Thus, we have
$$T(g)=T(g)\chi_{E_0}+T(g)\chi_{E_1\backslash E_0}\in \lv+L^\fz(\rn).$$
From this and \eqref{1207x}, we deduce that
\begin{align*}
\|T(g)\|_{(\lv,L^\fz(\rn))_{\theta,\fz}}
&\le\sup_{t\in(0,\fz)}t^{-\theta}\lf[\|T(g)\chi_{E_0}\|_{\lv}+
t\|T(g)\chi_{E_1\backslash E_0}\|_{L^\fz(\rn)}\r]\\
&\ls \sup_{t\in(0,\fz)}t^{-\theta}\lf[\|T(g_0)\|_{\lv}+
t\|T(g_1)\|_{L^\fz(\rn)}\r]\\
&\ls \sup_{t\in(0,\fz)}t^{-\theta}\lf[\|g_0\|_{\hv}+
t\|g_1\|_{L^\fz(\rn)}\r]\\
&\ls \|g\|_{(\hv,L^\fz(\rn))_{\theta,\fz}}.
\end{align*}
Therefore, the above claim holds true.

By this claim and Lemma \ref{l1230c}, we conclude that,
if $f\in (\hv,L^\fz(\rn))_{\theta,\fz}$, then
$T(f)$ belongs to $W\!L^{\wz p(\cdot)}(\rn)$, namely,
$f\in W\!H^{\wz p(\cdot)}(\rn)$. Thus, \eqref{1219e} holds true.
This finishes the proof of Theorem \ref{t-inter}.
\end{proof}

We end this section by giving the proof of Corollary \ref{cor} via using Theorem
\ref{t-inter} and Lemma \ref{l1230c}.

\begin{proof}[Proof of Corollary \ref{cor}]
Since $p_-\in(1,\fz)$, it follows from \cite[Lemma 3.1]{ns12} that
\begin{equation}\label{1212b}
\hv=\lv
\end{equation}
with equivalent norms. Moreover, there exists $\theta\in(0,1)$ such that
$(1-\theta)p_-\in(1,\fz)$. By this, \eqref{1212b},
Theorem \ref{t-inter} and the fact that
$$(L^{(1-\theta)p(\cdot)}(\rn), L^\fz(\rn))_{\theta,\fz}=W\!L^{p(\cdot)}(\rn)$$
(see Lemma \ref{l1230c}), we conclude that
\begin{align*}
\whv&=(H^{(1-\theta)p(\cdot)}(\rn),L^\fz(\rn))_{\theta,\fz}\\
&=(L^{(1-\theta)p(\cdot)}(\rn),L^\fz(\rn))_{\theta,\fz}
=W\!L^{p(\cdot)}(\rn).
\end{align*}
This finishes the proof of Corollary \ref{cor}.
\end{proof}

\medskip

\noindent Ciqiang Zhuo

\medskip

\noindent Key Laboratory of High Performance Computing and Stochastic
Information Processing (HPCSIP) (Ministry of Education of China), College of Mathematics
and Computer Science, Hunan Normal University, Changsha, Hunan 410081, P. R. China

\smallskip

\noindent {\it E-mail}: \texttt{cqzhuo@mail.bnu.edu.cn} (C. Zhuo)

\bigskip

\noindent Dachun Yang (Corresponding author) and Wen Yuan

\medskip

\noindent  School of Mathematical Sciences, Beijing Normal University,
Laboratory of Mathematics and Complex Systems, Ministry of
Education, Beijing 100875, People's Republic of China

\smallskip

\noindent {\it E-mails}: \texttt{dcyang@bnu.edu.cn} (D. Yang)

\hspace{0.988cm} \texttt{wenyuan@bnu.edu.cn} (W. Yuan)


\begin{thebibliography}{100}

\bibitem{am02} E. Acerbi and G. Mingione, Regularity results for stationary
electro-rheological fluids, Arch. Ration. Mech. Anal. 164 (2002), 213-259.

\vspace{-0.3cm}

\bibitem{am05} E. Acerbi and G. Mingione,
Gradient estimates for the $p(x)$-Laplacean system,
J. Reine Angew. Math. 584 (2005), 117-148.

\vspace{-0.3cm}

\bibitem{aa16} A. Almeida and A. Caetano,
Atomic and molecular decompositions in variable exponent
2-microlocal spaces and applications,
J. Funct. Anal. 270 (2016), 1888-1921.

\vspace{-0.3cm}

\bibitem{ah10} A. Almeida and P. H\"ast\"o,
Besov spaces with variable smoothness and integrability,
J. Funct. Anal. 258 (2010), 1628-1655.

\vspace{-0.3cm}

\bibitem{bl76}
J. Bergh and J. L\"ofstr\"om, Interpolation Spaces. An introduction,
Grundlehren der Mathematischen Wissenschaften, No. 223. Springer-Verlag, Berlin-New York, 1976.

\vspace{-0.3cm}

\bibitem{c77} A.-P. Calder\'on, An atomic decomposition of distributions
in parabolic $H^p$ spaces, Adv. Math. 25 (1977), 216-225.

\vspace{-0.3cm}

\bibitem{cruz03} D. Cruz-Uribe, The Hardy-Littlewood maximal operator on variable-$L^p$
 spaces, in: Seminar of Mathematical Analysis (Malaga/Seville, 2002/2003), 147-156,
 Colecc. Abierta, 64, Univ. Sevilla Secr. Publ., Seville, 2003.

\vspace{-0.3cm}

\bibitem{cfbook} D. V. Cruz-Uribe and A. Fiorenza, Variable Lebesgue Spaces.
Foundations and Harmonic Analysis, Applied and Numerical Harmonic Analysis,
Birkh\"auser/Springer, Heidelberg, 2013.


\vspace{-0.3cm}

\bibitem{cf06} D. Cruz-Uribe, A. Fiorenza, J. M. Martell and C. P\'erez,
The boundedness of classical operators on variable $L^p$ spaces,
Ann. Acad. Sci. Fenn. Math. 31 (2006), 239-264.

\vspace{-0.3cm}

\bibitem{cw14} D. Cruz-Uribe and L.-A. D. Wang,
Variable Hardy spaces,
Indiana Univ. Math. J. 63 (2014), 447-493.

\vspace{-0.3cm}

\bibitem{din04} L. Diening, Maximal function on generalized Lebesgue
 spaces $\lv$, Math. Inequal. Appl. 7 (2004), 245-253.

\vspace{-0.3cm}

\bibitem{dhr11} L. Diening, P. Harjulehto, P. H\"ast\"o and
M. R{$\mathring{\rm u}$}\v{z}i\v{c}ka,
Lebesgue and Sobolev Spaces with
Variable Exponents, Lecture Notes in Math. 2017,
Springer, Heidelberg, 2011.
\vspace{-0.3cm}

\bibitem{dhr09} L. Diening, P. H\"ast\"o and S. Roudenko, Function spaces of
variable smoothness and integrability, J. Funct. Anal. 256 (2009), 1731-1768.

\vspace{-0.3cm}

\bibitem{fz01}
X. Fan and D. Zhao, On the spaces $L^{p(x)}(\Omega)$ and $W^{m,p(x)}(\Omega)$,
J. Math. Anal. Appl. 263 (2001), 424-446.
\vspace{-0.3cm}

\bibitem{frs74} C. Fefferman, N. M. Rivi\`ere and Y. Sagher,
Interpolation between $H^p$ spaces: the real method, Trans. Amer.
Math. Soc. 191 (1974), 75-81.

\vspace{-0.3cm}

\bibitem{fs86} R. Fefferman and F. Soria, The space weak $H^{1}$,
Studia Math. 85 (1986), 1-16.

\vspace{-0.3cm}

\bibitem{g09-1} L. Grafakos,
Classical Fourier Analysis, Third edition,
Graduate Texts in Math. 249, Springer, New York, 2014.

\vspace{-0.3cm}

\bibitem{gh16}
L. Grafakos and D. He, Weak Hardy spaces. Some topics in harmonic analysis and applications,
177-202, Adv. Lect. Math. (ALM), 34, Int. Press, Somerville, MA, 2016.

\vspace{-0.3cm}

\bibitem{He14}
D. He, Square function characterization of weak Hardy spaces,
J. Fourier Anal. Appl. 20 (2014), 1083-1110.

\vspace{-0.3cm}

\bibitem{ins14}
M. Izuki, E. Nakai and Y. Sawano, Function spaces with variable
exponents--an introduction--, Sci. Math. Jpn. 77 (2014), 187-315.

\vspace{-0.3cm}

\bibitem{kv14}
H. Kempka and J. Vyb\'iral,
Lorentz spaces with variable exponents,
Math. Nachr. 287 (2014), 938-954.

\vspace{-0.3cm}

\bibitem{kr91} O. Kov\'a\v{c}ik and J. R\'akosn\'{\i}k, On spaces
$L^{p(x)}$ and $W^{k,p(x)}$, Czechoslovak Math. J. 41 (116) (1991), 592-618.

\vspace{-0.3cm}

\bibitem{liu91} H. Liu,
The weak $H^p$ spaces on homogenous groups,
in: Harmonic Analysis (Tianjin, 1988), 113-118, Lecture Notes in
Math. 1494, Springer, Berlin, 1991.

\vspace{-0.3cm}

\bibitem{ns12} E. Nakai and Y. Sawano, Hardy spaces with variable exponents
and generalized Campanato spaces, J. Funct. Anal. 262 (2012), 3665-3748.

\vspace{-0.3cm}

\bibitem{or31} W. Orlicz, \"Uber konjugierte Exponentenfolgen, Studia Math.
3 (1931), 200-211.

\vspace{-0.3cm}

\bibitem{rs73}
N. M. Rivi\`ere and Y. Sagher, Interpolation between $L^\fz$ and $H^1$, the real method,
J. Funct. Anal. 14 (1973), 401-409.

\vspace{-0.3cm}

\bibitem{Ru91}
W. Rudin, Functional Analysis. Second edition,
International Series in Pure and Applied Mathematics, McGraw-Hill, Inc., New York, 1991.

\vspace{-0.3cm}

\bibitem{Sa13} Y. Sawano, Atomic decompositions of Hardy spaces with variable exponents
and its application to bounded linear operators,
Integral Equations Operator Theory 77 (2013), 123-148.

\vspace{-0.3cm}

\bibitem{StTo89}
J.-O. Str\"omberg and A. Torchinsky, Weighted Hardy Spaces,
Lecture Notes in Math. 1381, Springer-Verlag, Berlin, 1989.

\vspace{-0.3cm}

\bibitem{Xu08} J. Xu, Variable Besov and Triebel-Lizorkin spaces,
Ann. Acad. Sci. Fenn. Math. 33 (2008), 511-522.

\vspace{-0.3cm}

\bibitem{yyyz16}
X. Yan, D. Yang, W. Yuan and C. Zhuo, Variable weak Hardy spaces and their applications,
J. Funct. Anal. 271 (2016), 2822-2887.

\vspace{-0.3cm}

\bibitem{yyz16}
D. Yang, W. Yuan and C. Zhuo, A survey on some variable function spaces,
in: Function Spaces and Inequalities, Springer, 2017 (to appear).

\vspace{-0.26cm}

\bibitem{yzn16} D. Yang, C. Zhuo and E. Nakai,
Characterizations of variable exponent Hardy spaces via Riesz transforms,
Rev. Mat. Complut. 29 (2016), 245-270.

\vspace{-0.3cm}

\bibitem{yzy15} D. Yang, C. Zhuo and W. Yuan, Triebel-Lizorkin type
spaces with variable exponents, Banach J. Math. Anal. 9 (2015),  146-202.

\vspace{-0.3cm}

\bibitem{yzy151} D. Yang, C. Zhuo and W. Yuan, Besov-type spaces with variable
smoothness and integrability, J. Funct. Anal. 269 (2015), 1840-1898.

\vspace{-0.3cm}

\bibitem{zyl16} C. Zhuo, D. Yang and Y. Liang, Intrinsic square function characterizations of
Hardy spaces with variable exponents, Bull. Malays. Math. Sci. Soc. 39 (2016), 1541-1577.
\end{thebibliography}
\end{document}